\documentclass[a4paper,12pt]{article}

\usepackage{amsmath,amsthm,amssymb,cite,enumitem, graphicx,geometry,todonotes}

\newtheorem{theorem}{Theorem}

\numberwithin{theorem}{section}

\newtheorem{corollary}[theorem]{Corollary}
\newtheorem{lemma}[theorem]{Lemma}

\newtheorem{proposition}[theorem]{Proposition}

\theoremstyle{definition}

\newtheorem{definition}[theorem]{Definition}
\newtheorem{example}[theorem]{Example}

\newcommand{\R}{\mathbb{R}}

\newcommand{\N}{\mathbb{N}}
\newcommand{\Q}{\mathbb{Q}}

\DeclareMathOperator{\aff}{aff}

\DeclareMathOperator{\cl}{cl}
\DeclareMathOperator{\cone}{cone}
\DeclareMathOperator{\core}{core}
\DeclareMathOperator{\co}{co}

\DeclareMathOperator*{\dbigcup}{\dot \bigcup}

\DeclareMathOperator{\lbd}{lbd}
\DeclareMathOperator{\lcl}{lcl}
\DeclareMathOperator{\lin}{lin}
\DeclareMathOperator{\linspace}{linspace}

\DeclareMathOperator{\lspan}{span}

\DeclareMathOperator{\icr}{icr}
\DeclareMathOperator{\pri}{pri}

\newcommand{\proplhd}{%
  \mathrel{\ooalign{$\lneq$\cr\raise.22ex\hbox{$\lhd$}\cr}}}

\title{The intrinsic core and minimal faces of convex sets in general vector spaces}
\author{R. D\'iaz Mill\'an\thanks{Deakin University, Melbourne, Australia} { and} Vera Roshchina\thanks{UNSW Sydney, Australia}}

\begin{document}

\maketitle


\begin{abstract} Intrinsic core generalises the finite-dimensional notion of the relative interior to arbitrary (real) vector spaces. Our main goal is to provide a self-contained overview of the key results pertaining to the intrinsic core and to elucidate the relations between intrinsic core and facial structure of convex sets in this general context.

We gather several equivalent definitions of the intrinsic core, cover much of the folklore, review relevant recent results and present examples illustrating some of the phenomena specific to the facial structure of infinite-dimensional sets.
\end{abstract}

\medskip


\medskip


\section{Introduction}

The relative interior of a convex set in a finite-dimensional real vector space is the interior of this convex set relative to its affine hull; a direct generalisation of this notion to real vector spaces is the intrinsic core,  introduced in \cite{kleepart1} and developed in \cite{Holmes} (also called pseudo-relative interior \cite{BorweinGoebel}, the set of inner points \cite{RelintInnerPoints} and the set of weak internal points \cite{DYE1992983}). 

In contrast to the relative interior in finite dimensions, intrinsic core may be empty for fairly regular sets, which was a key motivation for introducing the notion of quasi-relative interior \cite{BorweinLewisPartiallyFinite} in the context of topological vector spaces. Even though the quasi-relative interior is a very interesting mathematical object (for instance, \cite{ZalinescuThreePb, ZalinescuOnTheUse}  study duality and separation in the context of quasi-relative interior and resolve a number of open questions), and has much practical importance, we find that limiting our study to the interplay between the intrinsic core and facial structure already offers rich material and provides a valuable perspective on the structure of convex sets in general vector spaces. Such a narrowly focused overview is long overdue: the well-known results and examples discussed here are scattered in the literature, and it may be difficult to find neat references for widely known statements. Besides, we were able to provide some new insights 
and discuss several examples of interesting infinite-dimensional convex sets from the perspective of facial structure (for instance, see Example~\ref{eg:Hilbert}, where we construct uncountable chains of faces of the Hilbert cube).

Our take on this topic is consistent with the approach of several recent papers that focus on the intrinsic core. For instance, \cite{UnitBall, InnerStructure, MinExp} study the facial structure of convex sets in linear vector spaces with relation to the intrinsic core, \cite{cones} is specifically focused on the intrinsic core of convex cones, and \cite{DangEtAl} reviews the basic properties of a (less general) notion of the algebraic core, with an emphasis on separation. Notably \cite{cones} and \cite{DangEtAl} do not mention the facial structure that is central to our exposition.  

We begin with a discussion on faces of convex sets in Section~\ref{sec:faces}, obtaining a new characterisation of minimal faces in terms of the lineality subspace of the cone of feasible directions (Proposition~\ref{prop:minface}), leading to a (known) characterisation of a face as an intersection of line segments (Corollary~\ref{cor:minfaceunion}). We then talk about chains of faces and present an example of a convex set with uncountable chains of faces (Example~\ref{eg:Hilbert}).

In Section~\ref{sec:intrinsic} we discuss several equivalent definitions of intrinsic core, and focus on characterising the intrinsic core in terms of minimal faces: specifically, in Proposition~\ref{prop:charicore} we show that the intrinsic core of a convex set coincides with the set of points for which the minimal face of these points is the entire set $C$, using Proposition~\ref{prop:minface}. We also provide an extensive list of properties of the intrinsic core, with self-contained proofs and references. We end the section with the decomposition result (a convex set is a disjoint union of the intrinsic cores of its faces, Theorem~3.14).

In Section~\ref{sec:algcl} we discuss the notions of linear closure and boundary and relate these to the intrinsic core and faces of convex sets. We also talk about separation and supporting hyperplanes in general real vector spaces, outlining classic separation results pertaining to the setting of general vector spaces, and highlighting the role of intrinsic core in convex separation.We then provide a brief summary in Section~\ref{sec:summary}.


\section{Convex sets and their faces}\label{sec:faces}

Recall that a subset $C$ of a real vector space $X$ is convex if for any two points $x,y\in C$ we have $[x,y]\subseteq C$, where $[x,y]$ is the line segment connecting the points $x$ and $y$,
$$
[x,y] = \{\alpha x+ (1-\alpha) y\, |\, \alpha \in [0,1]\}.
$$ 

We will also use the notation $(x,y)$ to denote the open line segments connecting $x$ and $y$, also $[x,y)$ and $(x,y]$ for the relevant half-open segments. Note that algebraically it makes perfect sense to consider degenerate line segments of the form $(x,x) = (x,x] = [x,x)= [x,x]$, since in this case 
\[
(x,x) = \{t x+ (1-t) x\,|\,  t\in (0,1)\} = \{x\}.
\]
Even though this notation may appear confusing geometrically, it allows to treat the endpoints of all line segments including the degenerate singleton ones in a unified fashion, and hence streamline much of our discussion.

Throughout the paper we assume that $X$ is a real vector space, and that all the sets we consider live in this space (unless stated otherwise). We reiterate this assumption in some of the statements for the ease of reference.

A convex subset $F$ of a convex set $C$ is called a {\em face} of $C$ if for every $x\in F$ and every $y,z \in C$ such that $x\in (y,z)$, we have $y,z\in F$. The set $C$ itself is its own face, and the empty set is a face of any convex set $C$.  A face $F$ of $C$ is \emph{proper} if $F$ is nonempty and does not coincide with $C$. We write $F\lhd C$ for the faces $F$ that do not coincide with $C$ and $F\unlhd C$ for all faces $F$ of $C$.  Faces that consist of one point only are called \emph{extreme points}. Some faces of two-dimensional convex sets are shown schematically in Figure~\ref{fig:faces}.
\begin{figure}[ht]
    \centering
    \includegraphics[width=0.8\textwidth]{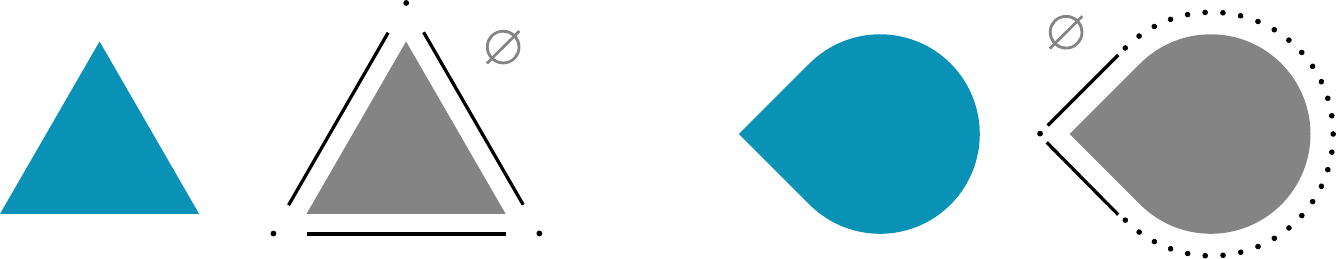}
    \caption{Faces of two-dimensional convex sets: proper faces are shown in black.}
    \label{fig:faces}
\end{figure}

The following two statements are well-known and follow from the definition of a face.

\begin{lemma}\label{lem:faceofaface}
Let $C$ be a convex subset of a real vector space $X$. If $F\unlhd C$ and $E\subseteq F$, then $E\unlhd F$ if and only if $E\unlhd C$.
\end{lemma}
\begin{proof} Suppose that $C$ is a convex set, $F\unlhd C$ and $E\subseteq F$. By the definition of a face, $E\unlhd C$ is equivalent to having for any $x,y \in C$ with  $(x,y)\cap E \neq \emptyset$ that $x,y\in C$. However since $F$ is a face of $C$, and $E\subseteq F$, any such pair $x,y$ must also be in $F$ (and vice versa, as $x,y\in F\subset C$). Hence, $E\unlhd C$ is equivalent to $E\unlhd F$. 
\end{proof}

\begin{lemma}\label{lem:intersection} Let $\mathcal{F}$ be a collection of faces of a convex set $C\subseteq X$, that is, every member of the set $\mathcal{F}$ is a face of $C$. Then
$$
E := \bigcap_{F\in \mathcal{F}} F
$$
is also a face of $C$. 
\end{lemma}
\begin{proof} First notice that $E$ is a convex subset of $C$. Now, if for some $x,y\in C$ the open line segment $(x,y)$ intersects $E$, then it intersects each $F \in \mathcal{F}$, and hence $[x,y]\subseteq E$, so $E$ is a face of $C$ by definition.
\end{proof}

\subsection{Minimal faces}

For any subset $S$  of a convex set $C$ there exists a unique minimal (in terms of the set inclusion) face $F\unlhd C$ that contains $S$. We can define minimal faces of a set in a constructive way, with the help of Lemma~\ref{lem:intersection}.

Let $S\subseteq C$, where $C$ is a  convex subset of  a real vector space $X$. The {\em minimal face}\index{minimal face}  of $C$ containing $S$ or just the minimal face of $S$ in $C$, is 
$$
F_{\min} (S,C) := \bigcap \{ F\,|\, F\unlhd C, S\subseteq F \}.
$$	

The set $F_{\min} (S,C)$ is a face due to Lemma~\ref{lem:intersection}, and it is the smaller face (with respect to set inclusion) that contains $S$. When $S=\{x\}$ is a singleton we use the notation $F_{\min}(x,C) = F_{\min}(\{x\},C)$. 

By $\cone C$ we denote the conic hull of $C\subseteq X$, that is, the set of all finite nonnegative combinations of points in $C\subseteq X$.
\[
\cone C = \left\{\sum_{i\in I} \alpha_i x_i\, |\, x_i\in C, \, \alpha_i \geq 0\quad \forall i \in I, \; |I|<\infty \right\}.
\]
The conic hull $\cone C$ of any set $C\subseteq X$ is a convex cone (it is convex and positively homogeneous, $\lambda x\in K$ for all $x\in K$ and $\lambda>0$). When $C$ is convex, we have $\cone C = \R_{+} C = \{\alpha x\,|\, x\in C, \alpha \geq 0\}$. In particular, when $C$ is convex and $x\in C$, then $\cone (C-x)$ is the cone of feasible directions of $C$ at $x$, that is, it consists of the rays along which one can move away from $x$ while staying within $C$.
 
The \emph{lineality space} $\linspace K$ of a nonempty convex cone $K\subseteq X$ is the largest linear subspace contained in $K$ (the lineality space can be defined for general convex sets, but we will only need this notion for cones). The following facts are well-known.

\begin{proposition}\label{prop:lineality} Let $K$ be a convex cone, $K\subseteq X$. Then
\begin{enumerate}[label=(\roman*)]
    \item for any $x\in K$, we have  $x\in \linspace K$ if and only if $-x\in K$; \label{item:linchar}
    \item the lineality space $\linspace K$ is nonempty if and only if $0\in K$. \label{item:nonemptylin}
    \end{enumerate}
\end{proposition}
\begin{proof} The part \ref{item:nonemptylin} is straightforward: since $0$ must belong to every linear subspace, if $0\notin K$, then no linear subspace is contained in $K$. Likewise, if $0\in K$, then $\{0\}$ is a trivial linear subspace contained in $K$, and the lineality space can only be larger than that.

The necessary part of \eqref{item:linchar} is evident: assuming $x\in \linspace K\subseteq K$, and recalling that $\linspace K$ is a linear subspace, we must have $-x\in \linspace K$. 

To show the sufficiency, that is, if $x,-x\in K$, then $x\in \linspace K$, first observe that it follows from convexity that $0\in [-x,x]\subset K$, and hence by \ref{item:linchar} $ \linspace K \neq \emptyset$. Assume that nevertheless $x\notin \linspace K$. Let 
\[
L = \lspan (\linspace K\cup\{x\}).
\]
For any $w\in L$ we have $w = v + u$, with $v\in \linspace K\subseteq K$ and $u = \lambda x$ for some $\lambda \in \R$. Now $\lambda x \in K$, irrespective of the sign of $\lambda$, since $K$ is a cone and both $x$ and $-x$ are in $K$.  Now $u,v\in K$, hence $w = u+v\in K$ (as $u+v = 2 (\frac 1 2 u + \frac 1 2 v) \in K$ by convexity and homogeneity), and we have demonstrated that $L$ is a linear subspace contained in $K$ that is strictly larger than $\linspace K$, which is impossible.
\end{proof}

It follows from \ref{item:linchar} that the lineality space of a convex cone is uniquely defined, since it is fully characterised by the set of points that belong to the cone along with their negatives.

Also recall that the \emph{Minkowski sum} of two (convex) sets $C,D\subseteq X$ is 
\[
C+ D = \{x+y\, |\, x\in C, y\in D\}.
\]
When one of the sets is a singleton, we abuse the notation and write $C+x$ for $C+\{x\}$. 
We show next that the minimal face of any point,  can be characterised in terms of the linearity space of the cone of feasible directions of the point.

\begin{proposition}\label{prop:minface} Let $C\subseteq X$ be a nonempty convex set and let $x\in C$. Then 
\[
F_{\min} (x,C) = C\cap (\linspace \cone (C-x)+x).
\]
\end{proposition}

Before moving on to the proof of Proposition~\ref{prop:minface}, we state and prove the following elementary technical proposition that will be used as a building block in several proofs.

\begin{proposition}\label{prop:technical} Let $a,b,u,v\in X$. If $a\in (u,v)$, then for any $c\in (a,b)$ there exists $w\in (v,b)$ such that $c\in (u,w)$.
\end{proposition}
\begin{proof} Let $a,b,u,v\in X$, and assume that $a\in (u,v)$ and $c\in (a,b)$ (see Fig.~\ref{fig:technical01}).
\begin{figure}[ht]
    \centering
    \includegraphics[width=0.36\textwidth]{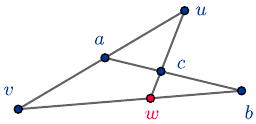}
    \caption{Illustration of the proof for Proposition~\ref{prop:technical}}
    \label{fig:technical01}
\end{figure}
There are $\alpha,\beta\in (0,1)$ such that 
\begin{equation}\label{eq:31575468}
a = \alpha u + (1-\alpha) v, \quad c = \beta a+ (1-\beta) b.
\end{equation}
We have from \eqref{eq:31575468} 
\begin{align}\label{eq:3243042}
c & = \beta (\alpha u + (1-\alpha) v)+ (1-\beta) b \notag\\
& = \alpha \beta u + (1-\alpha \beta) \left(\frac{\beta (1-\alpha)}{(1-\alpha \beta) } v + \frac{1-\beta }{(1-\alpha \beta) } b\right).
\end{align}
Let
\[
w = \frac{\beta (1-\alpha)}{1-\alpha \beta} v + \frac{1-\beta}{1-\alpha \beta} b.
\]
It is evident that $w\in (b,v)$, while from \eqref{eq:3243042} we have $c\in (u,w)$.
\end{proof}

\begin{proof}[Proof of Proposition~\ref{prop:minface}] We first show that 
\[
 C\cap (\linspace \cone (C-x)+x)\subseteq F_{\min} (x,C).
\]
Let $y \in  C\cap (\linspace \cone (C-x)+x)$. Then $y = x+ u $ for some $u\in \linspace \cone (C-x)$, and  by Proposition~\ref{prop:lineality}~\ref{item:linchar} we have $-u\in \linspace \cone (C-x)\subseteq \cone (C-x)$. This means that there exists some $t>0$  such that $x-t u\in C$. We have 
\[
\frac{t}{1+t} (x+u) + \frac{1}{1+t}(x- tu) 
=x,
\]  
hence $x\in (x+u, x-t u)$ and for every face $F$ of $C$ containing $x$ we must have $y = x+u\in F$. We conclude that

\[
y\in \bigcap_{\substack{F\unlhd C\\x\in F}} F = F_{\min}(x,C).
\]
It remains to show the converse,  that 
\[
F_{\min} (x,C) \subseteq C\cap (\linspace \cone (C-x)+x).
\]
Since $F_{\min}(x,C)\subseteq C$, it is then sufficient to show that 
\begin{equation}\label{eq:543545}
F_{\min} (x,C) = F_{\min}(x,C) \cap (\linspace \cone (C-x)+x).
\end{equation}
If we  prove that 
\begin{equation}\label{eq:453453}
F =  F_{\min}(x,C) \cap (\linspace \cone (C-x)+x)
\end{equation}
is a face of $C$, then from $x\in F$, we must have $F\subseteq F_{\min}(x,C)$, hence $F= F_{\min}(x,C)$, and so \eqref{eq:543545} holds.  

Let $y\in F$ and $u,v\in C$ and $\alpha \in (0,1)$ be such that 
\begin{equation}\label{eq:y001}
y = (1-\alpha) u + \alpha v.
\end{equation}
Our goal is to show that $u\in F$.

Let $p = y-x$. Since $y\in \linspace \cone (C-x)+x$, we have $p\in \linspace \cone (C-x)$, and by Proposition~\ref{prop:lineality} \ref{item:linchar} we have $-p\in \linspace \cone (C-x)\subseteq \cone (C-x)$. There exists $t>0$ such that $z: = x-t p \in C$. Applying Proposition~\ref{prop:technical} to the points $y,z, u,v$ and $x$ (cf. Figs.~\ref{fig:technical01} and~\ref{fig:technical02}), 
\begin{figure}[ht]
    \centering
    \includegraphics[width=0.55\textwidth]{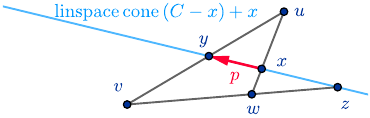}
    \caption{The two-dimensional gadget used in the proof of Proposition~\ref{prop:minface}}
    \label{fig:technical02}
\end{figure}
we conclude that there exists some $w\in (v,z)$ such that $x\in (u,w)$, which yields $[u,w]\subseteq F_{\min}(x,C)$, hence, $u-x, -(u-x)\in \cone (C-x)$, so  $u-x\in \linspace \cone(C-x) $. We conclude that $u\in \linspace \cone (C-x)+x$. Hence, from \eqref{eq:453453} we conclude that $u\in F$.

\end{proof}

It follows from Proposition~\ref{prop:minface} that the minimal face $F_{\min}(x,C)$ of $x\in C$ consists precisely of the line segments in $C$ that contain $x$ in their interiors (cf. Proposition~2.3 item 7 of \cite{FacelessProblem}).

\begin{corollary}\label{cor:minfaceunion}\label{linesunion} The minimal face $F_{\min}(x,C)$ for $x\in C$ (where $C$ is a convex subset of a real vector space $X$) can be represented as
\begin{equation}\label{minface:union}
F_{\min}(x,C) = \bigcup \{ [y,z]\subseteq C, \, x\in (y,z)\}.
\end{equation}
\end{corollary}
\begin{proof} It is evident that the set on the right-hand side is a subset of any face that contains $x$, hence, it must be contained in the minimal face of $x$ in $C$ as well. It remains to show that the minimal face of $x$ in $C$ is a subset of the right-hand side. Let $y\in F_{\min}(x,C)$. Then by Proposition~\ref{prop:minface} we have 
\[
y \in \linspace \cone (C-x)+x,
\]
equivalently $y-x\in \linspace \cone (C-x)$. This means that $x-y\in \linspace \cone (C-x)\subseteq \cone (C-x)$, and so there exists some $t>0$ such that $t(x-y)\in C-x$, equivalently $x+ t(x-y)\in C$. It remains to note that
\[
x = \frac{t}{1+t} y + \frac{1}{1+t} (x+ t(x-y)), 
\]
hence $x\in (y, x+ t(x-y))$, and so $y$ belongs to the right-hand side of \eqref{minface:union}.
\end{proof}



\subsection{Chains of faces}

Recall that a \emph{chain} is a totally ordered set. A chain $\mathcal{F}$ of faces of a convex set $C$ is such a set of faces of $C$ that for every $F, E\in \mathcal{F}$ we have either $E\subsetneq F$ or $F\subsetneq E$ (Lemma~\ref{lem:faceofaface} implies that in this case either $E\lhd F$ or $F\lhd E$).

In the finite-dimensional setting a chain of faces of a convex set is always finite, since the affine hulls generated by strictly nested faces must have strictly increasing dimensions, and there is a natural bound coming from the dimension of the ambient space (the nesting of affine hulls also works in the infinite dimensional case, see Proposition~\ref{prop:genhulls} below). The length of any chain of faces in $\R^n$ is at most $n+2$, since a chain of faces may include the empty set as well as an extreme point, which is zero-dimensional (see Fig.~\ref{fig:tetrahedron}).
\begin{figure}[ht]
    \centering
    \includegraphics[width=0.9\textwidth]{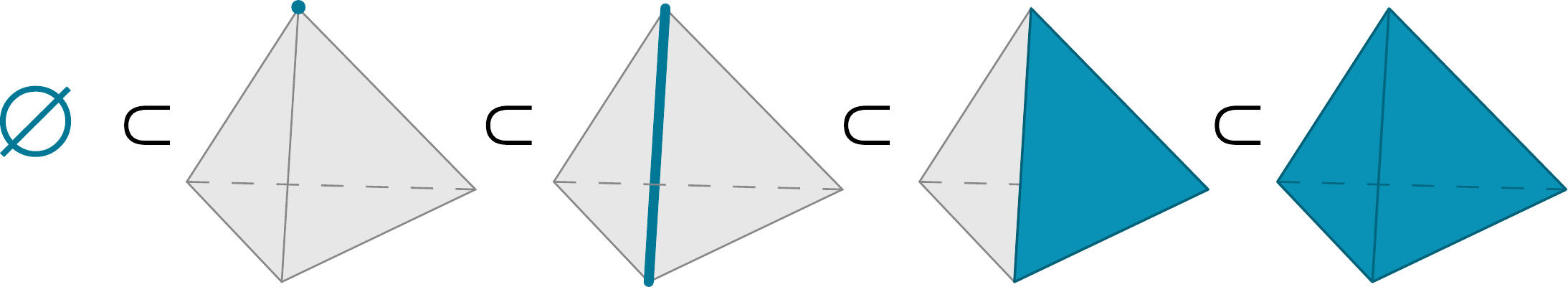}
    \caption{A maximal chain of faces of a tetrahedron.}
    \label{fig:tetrahedron}
\end{figure}
Polytopes have chains of faces of maximal possible length $d+2$, where $d$ is the intrinsic dimension of the polytope. In the infinite dimensional setting however there are convex sets that may have not only infinite, but uncountable chains of faces. An example of a set with uncountable chains of faces is the  Hilbert cube discussed next.

\begin{example}[Hilbert cube] \label{eg:Hilbert}Consider the subset $C$ of $l_2$ defined as the infinite product of diminishing line segments, 
$$
C:= \prod_{n\in \N^*}\left[0,\frac{1}{n}\right]
$$
(here $\N^* = \{1,2,3,\dots\}$ is the set of natural numbers that doesn't contain 0). It is not difficult to observe that the subsets of $C$ defined by
$$
\prod_{n\in \N^*}s_n, \quad s_n \in \left\{\{0\}, \left\{\frac{1}{n}\right\}, \left[0,\frac{1}{n}\right]\right\}
$$
are faces of $C$. We can now construct an uncountable chain of faces using a well-known trick. There exists an uncountable chain $\mathcal C$ in the power set  $\mathcal{P}(\N^*)$ of natural numbers (to see this consider any bijection between $\N^*$ and $\Q$, and then construct the $\phi:\R\to \mathcal{P}(\Q)$ that maps a real number $t$ to the set of all rational numbers that are strictly smaller than $t$: this mapping generates the required chain).

For any $c\in \mathcal{P}(\N^*)$ we let 
$$
F_{c}:= \prod_{n\in \N^*} s_n, \quad s_n = \begin{cases}
\left[0, \frac{1}{n}\right], & n \in c,\\
[0], & n \notin c.
\end{cases} 
$$

It is evident that if $c_1\subset c_2$ for some $c_1,c_2\in \mathcal C$, then $F_{c_1}\subset F_{c_2}$, hence, we have constructed an uncountable chain of faces $\{F_c\}_{c\in \mathcal C}$.

\end{example}

Notice that in Example~\ref{eg:Hilbert} we didn't need to make the `edge length' of the hypercube diminishing in $n$, and could have considered a hypercube in the general sequence space, say with a unit edge length. However having an example in a well-behaved Hilbert space like $l_2$ demonstrates that the phenomenon occurs naturally in a fairly standard setting.

 Another instance of a set with uncountable chains of faces is discussed in Section~\ref{sec:algcl}, its construction relies on an infinite Hamel basis (see Example~\ref{eg:ubiquitous}). 

\begin{proposition}\label{prop:unionchain} Let $C$ be a convex set in a real vector space $X$ and suppose that $\mathcal{F}$ is a chain of its faces. Then 
\[
E = \bigcup_{F\in \mathcal{F}} F
\]
is a face of $C$.
\end{proposition}
\begin{proof} First observe that $E$ is a convex subset of $C$: every point $x\in E$ belongs to some face of $C$, and hence belongs to $C$; furthermore, for any $u,v\in E$ there exist $F_u, F_v\in \mathcal{F}$ such that $u\in F_u$ and $v\in F_v$. Since $\mathcal{F}$ is a chain, without loss of generality we can assume that $F_u\subset F_v$, and hence $u,v\in F_v$. Therefore, $[u,v]\subseteq F_v \subseteq C$, and so $C$ is convex. 

Let $x\in E$, and let $y,z\in C$ be such that $x\in (y,z)$. Since $E$ is a union of faces of $F$, there must be a face $F_x\in \mathcal{F}$ such that $x\in F_x$. Since $F_x$ is a face of $C$, we have 
\[
[y,z] \subseteq F_x \subseteq E,
\]
and so $E$ is a face of $C$.
\end{proof}

Every chain of faces generates an associated set of affine subspaces, which are the affine hulls of these faces. Recall that for any (convex) subset $C$ of a real vector space $X$ we define the \emph{affine hull} of $C$ as 
\[
\aff C = \left\{\sum_{i\in I} \alpha_i x_i,\, |\, \sum_{i\in I} \alpha_i = 1, x_i \in C\, \forall i \in I, |I|<\infty\right\}.
\]
Equivalently $\aff C $ is the smallest affine subspace containing $C$ (a set of the form $x+ L$, where $L$ is a linear subspace, and $x\in C$), or $\aff C = x+ \lspan (C-x)$.

The next proposition shows that chains of faces generate chains of affine subspaces, which results in a finite bound on chains of faces in finite dimensions, as was discussed in the beginning of this section.  

\begin{proposition}\label{prop:genhulls} If $E,F \unlhd C$ and $E\subsetneq F$, then $\aff E \subsetneq \aff F$.
\end{proposition}
\begin{proof} It is evident from the definition of the affine hull that if $E\subseteq F$, then $\aff E \subseteq \aff F$. It is hence sufficient to show that if $E\subsetneq F$ then there exists some $u\in \aff F \setminus \aff E$. Assume the contrary, then in particular $F\subseteq \aff E$. Since $E\subsetneq F$, there exists $x\in F\setminus E$, which can be represented as
\[
x = \sum_{i\in I} \alpha_i u_i, \qquad \sum_{i\in I} \alpha_i = 1, \quad u_i\in E\quad \forall i \in I, \; |I|<\infty. 
\]
By convexity
\[
y = \sum_{i\in I} \frac{1}{|I|} u_i \in E.
\]
Now there must exist a $t>0$ such that 
\[
\frac 1 {|I|} \geq t \left(\alpha_i -\frac{1}{|I|}\right) \quad \forall i \in I.
\]
Let 
\[
\lambda_i = \frac 1 {|I|} - t \left(\alpha_i -\frac{1}{|I|}\right), \quad i \in I.
\]
Since $\lambda_i>0$ for all $i\in I$ and 
\[
\sum_{i\in I} \lambda_i = 1,
\]
we have  
\[
z = y + t(y-x) = \sum_{i\in I} \lambda_i u_i \in E\subset C,
\]
and since $y\in E\subseteq C$ and $y\in (x,z)\subset C$, by the definition of a face we must have $x\in E$, which contradicts our assumption. 
\end{proof}

\section{Intrinsic core}\label{sec:intrinsic}

In finite-dimensional spaces the relative interior of a convex set $C$ is usually defined as its interior with respect to the affine hull of $C$. Intrinsic core 
is a generalisation of this notion to the real vector spaces (although note that for topological vector spaces the relative interior is alternatively generalised as the interior relative to the topological closure of the affine hull \cite{BorweinGoebel}). 

\subsection{Definition and basic properties of the intrinsic core}

There are several equivalent definitions of the intrinsic core: we discuss all of them throughout this section, however we start with the one that is perhaps the most elegant (also see the original work of Klee \cite{kleepart1} where this approach first appears).

\begin{definition}[Intrinsic core via line segments]\label{def:segments} Let $C$ be a convex subset of a real vector space $X$. Then the intrinsic core of $C$ is
\[
\icr C = \{x\in C\, |\, \forall y \in C \, \exists z\in C \, \text{ such that } x\in (y,z)\}.
\]
\end{definition}
In other words, $x\in \icr C$ if and only if one can extend the line segment connecting $x$ to any point of the set $C$ \emph{beyond} the point $x$, while staying within $C$. This is also the definition used in \cite{DYE1992983} (called the set of weak internal points), in \cite{FacelessProblem} (the set of inner points) and also in \cite{BorweinGoebel} (the set of relatively absorbing points or pseudo-relative interior).

Note that is a difference in the interpretation of the definition of the intrinsic core between our exposition and other literature: for instance, the discussion in \cite[Section~1]{FacelessProblem}) implies that singletons must have empty intrinsic cores, for the lack of any line segments, which we find inconsistent. It follows from Proposition~\ref{prop:charicore} that whenever $C$ is a singleton $C = \{x\}$, we have $F_{\min}(x,C) = \{x\} = C$, and hence $\icr C = C$. 

The definition of the intrinsic core can also be extended to general subsets of $X$, as it is done in \cite{InnerStructure}: in this case the instinsic core of some $S\subseteq X$ is defined as all points $x\in S$ such that for any line $L$ through $x$ either there are some $u\neq v$ such that $x\in (u,v)\subseteq L\cap S$ or $L\cap S = \{x\}$. 

\begin{proposition}\label{prop:icrisconvex} The intrinsic core of a convex subset $C$ of a vector space $X$ is a convex set.
\end{proposition}
\begin{proof} Let $x,y\in \icr C$ and suppose that $z\in (x,y)$. Then for any $u\in C$ there exist $v,w\in C$ such that $x\in (u,v)$ and $y\in (u,w)$. Then we have for some $s,t,r\in (0,1)$
\[
x = u + s (v-u), \quad y = u + t (w-u), \quad z = x + r (y-x).
\]
Substituting the expressions for $x$ and $y$ into the one for $z$, we obtain
\[
z = u + (s (1-r) + t r)(q-u),
\]
where 
\[
q = \frac{s(1-r)}{s (1-r) + t r} v + \frac{t}{s (1-r) + t r}w \in C
\]
by convexity (see Fig.~\ref{fig:technincal03}).
\begin{figure}[ht]
    \centering
    \includegraphics[width=0.4\textwidth]{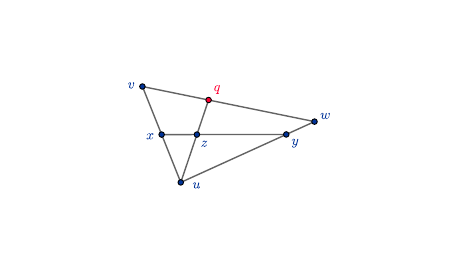}
    \caption{An illustration to the proof of Proposition~\ref{prop:icrisconvex}.}
    \label{fig:technincal03}
\end{figure}
Hence, $z\in (u,q)$ with $q\in C$, and by the arbitrariness of the choice of $u$ we have $z\in \icr C$. 
\end{proof}

In contrast to the finite-dimensional setting, where the relative interior of any nonempty convex set is also nonempty (e.g. see \cite[Theorem~6.2]{RockConvAn}), the intrinsic core of a convex set may be empty. The next example illustrates this phenomenon (cf. \cite[Theorem~5.2]{InnerStructure}). 

\begin{example}\label{eq:emptyicr} Let $C$ be a subset of $c_{00}$ (eventually zero sequences), 
\begin{equation}
C = \{x\in c_{00}\, |\, x_i \in [0,1]\}.\label{eq:8707987}
\end{equation}
First observe that $C$ is a convex set. We will show that for any $x\in C$ there exists $y\in C$ such that there isn't any $z\in C$ with $x\in (y,z)$, and hence $\icr C = \emptyset$.

Indeed, let $x = (x_1,\dots, x_k,\dots )\in C$. There is an  index $i\in \N$ such that $x_k = 0$ for every $k\geq i$. Let $y = (y_1,\dots, y_k, \dots)$ be such that 
\[
y_k = \begin{cases}
x_k, & k\in \N\setminus \{i\},\\
1, & k =i.
\end{cases}
\]
In other words, $y$ coincides with $x$ up to and including $i-1$-st entry, has 1 in the $i$-th position and zeros beyond $i$. If there was a $z\in C$ such that $x\in (y,z)$, then for some $t\in (0,1)$ $x = t y + (1-t) z$, and in particular $0 = x_i =  t y_i + (1-t) z = (1-t) + t z_i$, hence $z_i<0$, which is impossible by the definition of $C$. 
\end{example}

Convex sets that have nonempty intrinsic cores are called \emph{relatively solid} in \cite{cones}. Perhaps a better term would be to call such sets \emph{intrinsically solid}. It is tempting to think about sets with empty intrinsic cores as `small', however this intuition is misleading, as such sets can be very large, with their algebraic closure covering the entire space: we discuss this well-known phenomenon in more detail in Section~\ref{sec:algcl}.

Since an intrinsic core of some convex set may be empty, having $A\subset B$ for two convex sets $A$ and $B$ it doesn't yield $\icr A \subseteq \icr B$. Even though this behaviour already happens in finite-dimensional spaces, the general case is in a sense more extreme, since it may happen that the intrinsic core of $A$ is nonempty, and the intrinsic core of $B$ is empty. A trivial example is to consider some set $B$ such that $B\neq \emptyset$ and $\icr B = \emptyset$, and let $A=\{x\}$ for some $x\in B$.

We next prove a well-known technical result that effectively means that the intersection of a line passing through the intrinsic core of a convex set with this convex set lies in the intrinsic core, except possibly the endpoints of this intersection. We later prove a similar result involving the algebraic closure (Proposition~\ref{prop:calcintcore} \ref{calc:bdsegment}).

\begin{proposition}\label{prop:basicsegment} Let $C$ be a convex subset of a real vector space $X$. If $x\in \icr C$ and $y\in C$, then there exists $z\in \icr C$ such that $x\in (z,y)\subseteq \icr C$. In particular, $[x,y)\subset \icr C$. 
\end{proposition}
\begin{proof} 
Consider some $x\in \icr C$ and $y\in C$. Fix an arbitrary $q\in (x,y)$. For any $u\in C$ there exists $v\in C$ such that  $x\in (u,v)$ by the definition of the intrinsic core. Applying Proposition~\ref{prop:technical} to $x,y,u,v$ and $q$, we deduce that there exists $w\in (v,y)$ such that $q\in (u,w)$. Since $v,y\in C$, by convexity $w\in C$. Then $q\in \icr C$ by the definition of the intrinsic core. We have hence shown that $[x,y)\subseteq \icr C$. Now since $x\in \icr C$ and $y\in C$, there must be some $z\in C$ such that $x\in (z,y)$. Then by the earlier proven statement we must have $[x,z)\subset \icr C$. Hence $x\in (z,y)\subseteq \icr C$.
\end{proof}

\begin{proposition}\label{prop:basicicr} Let $C$ be a convex subset of a real vector space $X$. Then $\icr(\icr C)=\icr C$.
\end{proposition}
\begin{proof}
Observe that $\icr \icr C \subseteq \icr C$, so we only need to prove the converse. Now take any $x\in \icr C$. Then for any $y\in \icr C$ we also have $y\in C$. Since $x\in \icr C$ and $y\in C$, by Proposition~\ref{prop:basicsegment} there must be $z\in \icr C$ such that $x\in (z,y)\subseteq \icr C$. This shows that $x\in \icr \icr C$.
\end{proof}

Note that the statement of Proposition~\ref{prop:basicicr} is discussed in a more general setting of not necessarily convex sets in  Theorem 2.3 in \cite{InnerStructure}; even though the notion of intrinsic core can be generalised to arbitrary (non-convex) sets, the statement of Proposition~\ref{prop:basicicr} is not true for nonconvex sets (see \cite[Example~3.1]{RelintInnerPoints}).

In the next proposition we show that when a convex subset of a convex set has empty intrinsic core, the minimal face of this subset coincides with the minimal face of any point in the intrinsic core of the subset. This property is well-known in the finite dimensional case for the relative interior.

\begin{proposition}\label{prop:FminS} Let $S$ and $C$ be convex sets in a real vector space $X$. If $S\subseteq C$ and $\icr S \neq \emptyset$, then for any $x\in \icr S$ we have $F_{\min}(x, C) = F_{\min}(S,C)$. Moreover, $\icr S \subseteq \icr F_{\min}(S,C)$.
\end{proposition}
\begin{proof}
Under the assumptions of the proposition, suppose that $x\in \icr S$. Then by Proposition~\ref{prop:basicsegment} for any $y\in S$ there exists $z\in S $ such that $x\in (y,z)$. Since $S\subset C$, Corollary~\ref{cor:minfaceunion} then yields $S\subseteq F_{\min}(x,C)$. Hence the minimal face of $S$ must be a subset of $F_{\min}(x,C)$, however there can be no smaller face containing $x$ (and hence $S$), therefore, $F_{\min}(x,C) = F_{\min}(S,C)$.

Now for any $x\in \icr S$ we have $x\in \icr F_{\min}(x,C) = F_{\min}(S,C)$, hence, $\icr S \subseteq \icr F_{\min}(S,C)$.
\end{proof}

It doesn't look like anything similar to Proposition~\ref{prop:FminS} can be said about the minimal face $F_{\min}(S,C)$ when the subset $S$ of $C$ has empty intrinsic core. The minimal face of $S$ may or may not have empty intrinsic core in this case.

\subsection{Alternative definitions of the intrinsic core}

In \cite{BorweinGoebel} the  intrinsic core of a convex set $C\subseteq X$ (called the pseudo-relative interior, $\pri C$) is defined as follows.

\begin{definition}[Intrinsic core via the cone of feasible directions]\label{def:feasibled}
\[
\icr C=\{x\in C: \cone(C-x) \text{ is a linear subspace} \}.
\]
\end{definition}

We next show that this definition describes the same object as the one considered previously.

\begin{proposition}[{Lemma~2.3 in \cite{BorweinGoebel}}]\label{prop:lines} Let $C$ be a convex subset of a real vector space $X$. Then $\cone(C-x)$ is a linear subspace if and only if for every $u\in C$ there exists $v\in C$ such that $x\in (u,v)$.
\end{proposition}
\begin{proof} Since $C$ is a convex set, the set $\cone(C-x)$ is a linear subspace if and only if $\cone (C-x) = - \cone (C-x)$ (additivity holds naturally), equivalently for every $y = t (u-x)$, where $t>0$ and $u\in C$ there exists  $s>0$ and $v\in C$ such that $- y = s (v-x)$. This is in turn equivalent to having, for each $u\in C$, the existence of $v\in C$ and $r = s/t>0$ such that $u - x  = -r (v-x)$, which is equivalent to
\[
x =  \frac{1}{1+r} u + \frac{r}{1+r} v \in (u,v).
\]
\end{proof}

In a similar way to the intrinsic core we can also define the \emph{core} of a convex set (introduced by Klee \cite{kleepart1}, also see \cite{Holmes}). For some $x\in C$ we say that $x\in \core C$ if for every $y\in X\setminus \{x\}$ there exists $z\in (x,y)$ such that $[x,z]\subseteq C$. If $0\in \core C$, then $C$ is called \emph{absorbing}. Note that $\icr C = \core C$ if and only if either $\core C \neq \emptyset$ or $\icr C = \emptyset$.  We can  equivalently define the intrinsic core of a convex set $C$ as its core with respect to its affine hull. Notice that generally speaking the affine hull of a convex set is not necessarily a linear subspace of the ambient vector space $X$, since it may not contain zero. However as this space closed with respect to lines and segments, there is no impediment to defining the core of a convex set living in this affine subspace. 
\begin{definition}[Intrinsic core via the affine hull] \label{def:affine} Let $C$ be a convex set of a real vector space $X$. The intrinsic core of $C$ is the algebraic core of $C$ with respect to the affine hull of $C$, that is,
\[
\icr_X C = \core_{\aff C} C.
\]
\end{definition}

The next proposition ensures that this new definition aligns with the original Definition~\ref{def:segments}.

\begin{proposition}\label{prop:icrascore} Let $C$ be a convex subset of a real vector space $X$, and let $x\in C$. The following statements are equivalent:
\begin{enumerate}[label=(\roman*)]
    \item for every $y\in C$ there exists $z\in C$ such that $x\in (y,z)$;\label{technical:firstdef}
    \item for every $y\in \aff C$ there exists $z\in (x,y)$ such that $z\in C$. \label{technical:lastdef}
\end{enumerate}
\end{proposition}
\begin{proof} Let $x\in C$ and suppose \ref{technical:firstdef} holds. Take any $y\in \aff C$. Then 
\begin{equation}\label{eq:324324352354}
y = \sum_{i=1}^m \lambda_i u_i+ \lambda_0 x, \quad \sum_{i\in I} \lambda_i+\lambda_0 = 1, \quad u_i \in C\setminus \{x\}\;  \; \forall \, i\in \{1,\dots, m\}.
\end{equation}
If $m= 0$, then $y=x$ and \ref{technical:lastdef} holds with $z = x=y$. Otherwise for any $i\in \{1,\dots, m\}$  there exists $v_i\in C$ such that $x\in (u_i, v_i)$; explicitly, there is some $t_i\in (0,1)$ such that 
\begin{equation}\label{eq:324324352355}
x = t_i u_i + (1-t_i) v_i.
\end{equation}
Now let 
\begin{align*}
    \alpha_i := \begin{cases}
    \displaystyle\sum_{\substack{\lambda_i< 0,\\i\geq 1}}  \frac{\lambda_i}{t_i} +\lambda_0, & i = 0\\
    \lambda_i , & \lambda_i \geq 0 , i \geq 1,\\
      \frac{(t_i-1) \lambda_i}{t_i}, & \lambda_i < 0 , i \geq 1,\\
    \end{cases}
    \qquad
    w_i = \begin{cases}
    u_i, &  \lambda_i \geq 0,\\
    v_i, & \lambda_i <0.
    \end{cases}
\end{align*}
It is easy to see using \eqref{eq:324324352354} and \eqref{eq:324324352355} that the affine representation of $y$ can be rewritten in this new notation as 
\begin{equation}\label{eq:324324}
y = \sum_{i=1}^m \alpha_i w_i + \alpha_0 x.
\end{equation}
Here $\sum_{i=1}^m \alpha_i+ \alpha_0 = 1$, $\alpha_i\geq 0 $ and $w_i\in C$ for $i\in \{1,\dots, m\}$  Now if $\alpha_0\geq 0$, then $y\in C$, hence $[x,y]\subseteq C$ and for any $z\in (x,y)$ we have $[x,z]\subseteq C$.  Otherwise (if $\alpha_0<0$) let 
\[
z:= \frac{1}{1-\alpha_0} y+ \frac{-\alpha_0}{1-\alpha_0} x. 
\]
Observe that $z\in (x,y)$ and also substituting the expression for $y$ from \eqref{eq:324324} we have
\begin{align*}
z & 
= \sum_{i=1}^m \frac{\alpha_i}{1-\alpha_0} w_i.
\end{align*}
It is evident that $z\in C$, and therefore $[x,z]\subseteq C$. We have shown that \ref{technical:firstdef} yields \ref{technical:lastdef}. It remains to show the converse.

Assume now that \ref{technical:lastdef} holds, and let $y\in C$. Since $x,y\in C$, we have 
\[
x+ (x-y) \in \aff C. 
\]
By \ref{technical:lastdef} there exists $z\in (x, x+ (x-y))$ such that $z\in C$. Explicitly, there is some $t\in (0,1)$ such that
\[
z = (1-t) x+ t (x+ (x-y)) = (1+t) x - t y.
\]
Hence, 
\[
x = \frac{1}{1+t} z+ \frac{t}{1+t}y \in (z,y).
\]
This shows \ref{technical:firstdef}.
\end{proof}

\subsection{Intrinsic core and minimal faces}\label{ssec:intrinsic-minface}

Even though the intrinsic core of a convex set may be empty (as in  Example~\ref{eq:emptyicr}), it provides a disjoint decomposition of a convex set into the intrinsic cores of its faces. We prove this in Theorem~\ref{thm:UnionRiFaces} (cf. \cite[Corollary~2.4]{FacelessProblem}), but first we obtain a characterisation of the intrinsic core via minimal faces. Notice also that the next proposition was effectively proved in \cite[Proposition~2.3]{FacelessProblem}.

\begin{proposition}\label{prop:charicore} Let $C\subseteq X$ be a convex set, then 
\[
\icr C = \{x\in C, F_{\min}(x,C)=C\}.
\]
\end{proposition}
\begin{proof}  From the definition of intrinsic core we then have $x\in \icr C$ if and only if $\cone (C-x)$ is a linear subspace, this is equivalent to $\linspace \cone(C-x) = \cone (C-x)$.

If $x\in \icr C$ we then have  by Proposition~\ref{prop:minface} 
\[
F_{\min}(x,C) = C\cap (\linspace \cone (C-x) +x) = C \cap (\cone (C-x)+x) = C.
\]
Conversely, if $F_{\min} (x,C) = C$, we have 
\begin{equation}\label{eq:9808}
C = C \cap (\linspace \cone (C-x)+x). 
\end{equation}
In this case, if $\linspace \cone (C-x) \neq \cone (C-x)$, there must be a point $u\in \cone (C-x)\setminus \linspace \cone (C-x)$, and hence there's some $t>0$ such that $x+tu\in C$, but $x+tu \notin \linspace \cone (C-x)+x$, which contradicts \eqref{eq:9808}. 
\end{proof}

Proposition~\ref{prop:charicore} that we have just proved allows us to state yet another equivalent definition of the intrinsic core.

\begin{definition}[Intrinsic core via minimal faces]\label{def:viaminface} Let $C$ be a convex subset of a real vector space $X$. The intrinsic core of $C$ can be defined as follows,
\[
\icr C = \{x\in C, F_{\min}(x,C)=C\}.
\]
\end{definition}

\begin{example}[Revisiting Example~\ref{eq:emptyicr}]
Using the last definition of the intrinsic core, we can provide an alternative explanation on why the intrinsic core of the set $C$ in Example~\ref{eq:emptyicr} defined by  \eqref{eq:8707987}  is empty. 

For $ x\in C$ let 
\[
F_{x} = \{u\in C\,|\, u_i = 0 \, \forall i: \,  x_i =0\}.
\]
It is not difficult to observe that $F_{x}$ is a face of $C$, does not coincide with $C$, and contains $x$. Hence the minimal face $F_{\min}(x,C)\subseteq F_{x}$ is strictly smaller than $C$ for every $x\in C$. We conclude that $\icr C = \emptyset$. 
\end{example}

\begin{corollary}\label{cor:minfaceicr} Let $C$ be a convex subset of a real vector space $X$, and let $F$ be a face of $C$. Then $F= F_{\min}(x,C)$ if and only if $x\in \icr F$. 
\end{corollary}
\begin{proof} Observe that for $x\in F\unlhd C$ we have
\[
F_{\min}(x,F) = F_{\min}(x,C),
\]
hence, under the assumptions of this corollary, and using Proposition~\ref{prop:charicore}, we obtain 
\[
\icr F = \{x\in F\,|\, F_{\min}(x,F) = F\}  = \{x\in F\,|\, F_{\min}(x,C) = F\}  
\]
which shows the equivalence. 
\end{proof}

\begin{theorem}[cf. Corollary~2.4 in \cite{FacelessProblem}]\label{thm:UnionRiFaces} For a convex set $C\subset X$ and any $x\in C$ there is a unique face $F\unlhd C$ such that $x\in \icr F$. Consequently,
\begin{equation}\label{eq:disjoint}
C  = \dbigcup_{F\unlhd C} \icr F,
\end{equation}
where by $ \dbigcup$ we denote a disjoint union.
\end{theorem}
\begin{proof} We obtain this decomposition by observing that for every $x\in C$ there exists a unique minimal face $F_{\min}(x,C)$. Evidently $F_{\min}(x,F_{\min}(x,C)) = F_{\min}(x,C)$, hence by Proposition~\ref{prop:charicore} we have $x\in \icr F$. This shows that $C$ is the union of intrinsic cores of its faces. It remains to show that the union is disjoint. 

Suppose that $x\in \icr E \cap \icr F$, where $E$ and $F$ are different faces of $C$. Then we must have $F_{\min}(x,E) = E \neq F = F_{\min} (x,F)$. However $x\in E\cap F$ and hence  
\[
F_{\min}(x,E) = F_{\min} (x,F) = F_{\min}(x,C),
\]
a contradiction.
\end{proof}

The next statement is Theorem~2.7 from \cite{FacelessProblem}, rephrased in our notation. The author calls the problem of characterising convex sets with no proper faces the `faceless problem', hence the name of the theorem. It appears that this result follows from the decomposition of Theorem~\ref{thm:UnionRiFaces}.

\begin{theorem}[Faceless theorem]
A nonempty convex subset $C$ of a real vector space $X$ is free of proper faces if and only if $C = \icr C$.
\end{theorem}
\begin{proof}
Suppose that a nonempty convex set $C$ has no proper faces. Then the only nonempty face of $C$ is the set $C$ itself, and by Theorem~\ref{thm:UnionRiFaces} we have $C= \icr C$.

Conversely, assume that $C$ is a nonempty convex set such that $C = \icr C$. Assume that on the contrary $C$ has a proper face $F$. Since $F$ is nonempty, there is some point $x\in F$, and by Theorem~\ref{thm:UnionRiFaces} we can find a face $E\lhd F$ such that $x\in \icr E$. By Lemma~\ref{lem:faceofaface} the convex set $E$ is also a face of $C$, while by our construction $E\subsetneq C$. We conclude that $x\in \icr C \cap \icr E$, which contradicts the uniqueness claim of Theorem~\ref{thm:UnionRiFaces}. 
\end{proof}

Since only the faces that are minimal with respect to some point in the convex set feature in the decomposition \eqref{eq:disjoint}, only minimal faces `contribute' to the set. It is natural to expect that the faces that aren't minimal can be represented in some natural way via the minimal ones, for instance as the union of a chain, as is the case in Example~\ref{eq:emptyicr}. However this is not true, as is clear from a construction suggested by one of our referees. We come back to this construction in Section~\ref{sec:linclosure}. 

\subsection{Calculus of intrinsic cores}

We next gather several key calculus rules pertaining to intrinsic cores that haven't been covered earlier. Even though this section contains mostly well-known results, we reprove them for completeness of our exposition.

\begin{proposition}[{Lemma~3.6 (a) in \cite{BorweinGoebel}}]\label{prop:sum} If $A$ and $B$ are convex sets in a real vector space $X$, then \begin{equation}\label{eq:sumicrweak}
\icr A + \icr B \subseteq \icr(A+B).
\end{equation}
Moreover, if $\icr A \neq \emptyset$ and $\icr B \neq\emptyset$, then 
\begin{equation}\label{eq:sumicrsrtrong}
\icr A + \icr B = \icr(A+B).
\end{equation}
\end{proposition}
\begin{proof}
Note that if one of $\icr A$ or $\icr B$ is nonempty, then \eqref{eq:sumicrweak} holds trivially. To show that \eqref{eq:sumicrweak} holds in general, take $a\in \icr A$, $b\in \icr B$, and let $c = a+b$. Now pick any $z\in A+B$. We can represent $z$ as $z = x+y$, where $x\in A$ and $y\in B$. By the definition of intrinsic core there exist $u\in A$ and $v\in B$ such that $a\in (x,u)$ and $b\in (y,v)$. Algebraically this means that there exist $t,s>1$ such that
\[
u = x+ t(a-x), \quad v= y+ s(b-y).
\]
Let $\alpha: = \min\{s,t\}$. By convexity
\[
 x+ \alpha (a-x)\in (x,u)\subseteq A, \quad   y+ \alpha(b-y)\in (y,v)\subseteq B.
\]
Hence 
\[
w = (x+ \alpha (a-x))+(y+\alpha(b-y)) \in A+B.
\]
It is evident that $c\in (z,w)$, where $w\in A+B$. Since our choice of $z$ was arbitrary, we conclude that $c\in \icr (A+B)$.

Now assume that $\icr A,\icr B\neq \emptyset$. To show \eqref{eq:sumicrsrtrong}, in view of \eqref{eq:sumicrweak} it remains to demonstrate that 
\[
\icr(A+B)\subseteq \icr A + \icr B.
\]
Pick any point $c\in \icr (A+B)$. Since $\icr(A+B)\subseteq A+B$, there must be $a\in A$ and $b\in B$ such that $c = a+b$. Now choose any $x\in \icr A$ and $y\in \icr B$. Then $z = x+y\in A+B$. Now $z\in A+B$, $c\in \icr (A+B)$, hence there must be some $w\in A+B$ such that $c\in (z,w)\subseteq \icr (A+B)$.

Since $w\in A+B$, there must be $u\in A$, $v\in B$ such that $w=u+v$. Now by Proposition~\ref{prop:basicicr} we have $[x,u)\subseteq \icr A$ and $[y,v)\subseteq \icr B$. Now $c = (1-\alpha) z + \alpha w  $  for some $\alpha \in (0,1)$. Since $z = x+y$ and $w = u+v$, we conclude that  
\[
c = [(1-\alpha) x + \alpha u ] +  [(1-\alpha) y + \alpha v ]\in \icr A + \icr B.
\]
\end{proof}

\begin{corollary}\label{cor:addpoint} If $C$ is a convex subset of a real vector space $X$, then for any $x\in X$ 
\[
\icr (C+\{x\}) = \icr C + \{x\}.
\]
\end{corollary}
\begin{proof}
From Proposition~\ref{prop:sum} we have 
\[
\icr (C+\{x\}) \subseteq  \icr C  + \icr \{x\} = \icr C + \{x\},
\]
also since $C = (C+\{x\})-\{x\}$, 
\[
\icr C \subseteq  \icr (C+\{x\})-\{x\},
\]
hence 
\[
\icr C + \{x\} \subseteq \icr (C+\{x\}),
\]
and we have the required equality. 
\end{proof}

Note that in \cite[Proposition~2.5]{DangEtAl} it is shown that if $\core C = C$, then for any convex $D\subseteq X$ we have 
\[
\core (C+D) = C + D.
\]
This property doesn't generalise to intrinsic cores, even for finite dimensions. For instance, take $C = (0,1)\times \{0\}$, $D = \{0\}\times [0,1]$ in the plane. Then $\icr C = C$, but  
\[
\icr (C+D) = (0,1)\times (0,1) \neq (0,1)\times [0,1] = C+D.
\]

\begin{proposition}[cf. Lemma~3.3 in \cite{BorweinGoebel}]\label{prop:linmap} Let $C$ be a convex subset of a real vector space $X$, and let $A:X\to Y$ be a linear mapping from $X$ to another real vector space $Y$. Then 
\begin{equation}\label{eq:9089}
A \icr C \subseteq \icr (AC). 
\end{equation}
Moreover, if $A$ maps $C$ to $AC$ injectively or $\icr C \neq\emptyset$, then 
\begin{equation}\label{eq:lineareq}
A \icr C = \icr (AC).
\end{equation}
\end{proposition}
\begin{proof} Let $C$ be a subset of $X$, and suppose $A:X\to Y$ is a linear mapping. If $x\in \icr C$, let $u = A x$ and take any $v\in AC$. There must be a $y\in C$ such that $v = Ay$. Since $x\in \icr C$, by the definition of the intrinsic core there exists $z\in C$ such that $x\in (y,z)$. Hence $u\in (v, Az)$ and we conclude that $u\in \icr AC$. This proves  \eqref{eq:9089}.

We first prove that \eqref{eq:lineareq} holds for an injective linear mapping $A$: we only need to show the converse of \eqref{eq:9089}. So assume that $A$ maps $C\subseteq X$ to $AC\subseteq Y$ injectively. Consider any $x\in \icr AC$. There is a unique $u\in C$ such that $x = Au$. Take any $v\in C$, and let $y = Av$. Since $x\in \icr AC$, there must be $z\in AC$ such that $x\in (y,z)$, and hence there is some $w\in C$ such that $z = A w$. Now observe that $u\in (v,w)$. Indeed, we have $x = (1-t) y + t z$ for some $t\in (0,1)$, and hence $A u =  (1-t) A v  + A t w = A ((1-t) v + t w) $. We must have $u = (1-t) v + t w$, otherwise $A$ is not injective. We have therefore shown that $u\in (v,w)\subseteq C$, and by the arbitrariness of $v$ we conclude that $u\in \icr C$, hence $x = A u \in A \icr C$.

Now let's deal with the case when $A$ is not injective, but $\icr C \neq \emptyset$.  Let $L$ be the kernel of $A$, i.e. 
\[
L = \{x\in X\,|\, Ax  = 0 \}.
\]
Since $L$ is a linear subspace of $X$, it has a Hamel basis that can be completed to the entire space. Denote by $M$ the span of this complement, then for any $x\in X$ we have $x = x_L+ x_M$, where $x_L \in L$ and $x_M \in M$. 

Let $D$ be the projection of $C$ onto the complement $M$ of $L$, in other words,  
\[
D = \{x\in M\,|\, \exists y\in L: x+y \in C\}.
\]
It is evident that $D+L = C+L$, and therefore
\begin{equation}\label{eq:23402395243}
A C = A (C+L) = AD.
\end{equation}

Furthermore, observe that $D$ is the linear image of $C$ under the projection $P$ that takes an element of $X$ to the subset $M$ by trimming off all of the $L$-space coordinates. From \eqref{eq:9089} we have
\[
P \icr C \subseteq \icr (PC) = \icr D,
\]
and so $\icr D$ is nonempty. It is evident that  $\icr L = L \neq \emptyset$, hence we can apply Proposition~\ref{prop:sum}  to $\icr C$ and $L$ (and to $\icr D$ and $L$) to obtain 
\[
\icr C + L = \icr (C+L) = \icr (D+L) = \icr D + L, 
\]
we therefore have 
\begin{equation}\label{eq:24234234234}
A(\icr C)   = A(\icr C + L) = A(\icr D+L) = A (\icr D), 
\end{equation}
In view of \eqref{eq:23402395243} and \eqref{eq:24234234234} it remains to show that 
\[
\icr (AD) = A (\icr D).
\]
Since $A$ is an injective mapping on $D$,  the result is true by the previously proved claim.
\end{proof}

The following corollary of Proposition~\ref{prop:linmap} was proved in \cite{DangEtAl} for the algebraic core under the assumption of surjectivity of $A$ (see Lemma~4.1 in \cite{DangEtAl}). 

\begin{corollary}[{cf. \cite[Lemma~4.1]{DangEtAl}}] If $A:X\to Y$ is a surjective linear operator between two vector spaces $X$ and $Y$, and $C\subseteq X$, then 
\begin{equation}\label{eq:32423424}
A(\core C) \subseteq \core (AC),
\end{equation}
moreover if $\core C \neq \emptyset$, then \eqref{eq:32423424} holds as equality.
\end{corollary}
\begin{proof}
If $\core C = \emptyset$, then \eqref{eq:32423424} is a triviality. Assume that $\core C \neq \emptyset$. Then $\icr C = \core C$, and by Proposition~\ref{prop:linmap} we have 
\[
A(\core C)  = A(\icr C) = \icr (AC).
\]
It remains to show that $\icr (AC) = \core (AC)$, and for that it is sufficiently to show that $\core (AC)\neq \emptyset$. Take any $x\in \icr (AC)$ and $y\in Y$. Since $A$ is surjective, there must be some $u\in C$ and $v\in X$ such that $x = A u$ and $y = A v$. Since $\core A \neq \emptyset$, there must be some $w\in (u,v)$ suth that $[u,w]\subseteq C$. Hence $[z , x]\subseteq AC$, with $z = Aw \in (Au, Av) = (x,y)$, so $x\in \core (AC) \neq \emptyset$.

\end{proof}

\begin{corollary}[{Lemma 3.6 (c) in \cite{BorweinGoebel}}]\label{cor:lambda} If $C$ is a convex subset of a real vector space $X$, then  
\[
\icr \lambda C = \lambda \icr C \quad \forall \lambda \in \R\setminus \{0\}.
\]
\end{corollary}
\begin{proof}
Follows from $\lambda $ defining an injective linear mapping and Proposition~\ref{prop:linmap}.
\end{proof}

\begin{corollary} Let $C$ be a convex subset of a real vector space $X$ and let $D$ be a convex subset of a real vector space $Y$, then 
\[
\icr (C\times D)=\icr C\times\icr D.
\]
\end{corollary}
\begin{proof} First assume that $\icr C \neq \emptyset$ and $\icr D \neq \emptyset$, and define two linear mappings, $M:X\to X\times Y$ and $N:Y\to X\times Y$ as follows:
\[
M(x) = (x,0_Y), \quad N(y) = (0_X,y).
\]
These two mappings are injective, and therefore by Proposition~\ref{prop:linmap} we have
\[
M(\icr C) = \icr (MC),\quad N(\icr D) = \icr (ND).
\]
Now $C\times D = MC+ND$, therefore by Proposition~\ref{prop:sum} we have
\begin{align*}
\icr (C\times D) 
& = \icr (MC+ND) \\
& = \icr(MC)+\icr(ND) \\
& = M(\icr C)+N(\icr D)\\
& = \{(x,y)\, |\, x\in \icr C, y \in \icr D\}\\
& =\icr C \times \icr D.
\end{align*}

It remains to consider the case when one of the intrinsic cores $\icr C$ or $\icr D$ is empty and to show that the left-hand side $\icr (C\times D)$ must also be empty in this case. Without loss of generality assume that $\icr C = \emptyset$, but there is some $(x,y) \in \icr (C\times D)$. Pick any $a\in C$. The point $(a,y)$ must be in $C\times D$, hence, there exists $(b,c)\in C\times D$ and $t\in (0,1)$ such that
\[
(x,y) = t (a,y)+ (1-t) (b,c). 
\]
We conclude that $x = t a + (1-t) b$, where $b\in C$ and $t\in (0,1)$, and hence $x\in \icr C$, which contradicts the assumption.
\end{proof}

\section{Linear Closure}\label{sec:algcl}

In this section we discuss the notion of linear closure that can be defined using the natural topology of the real line. Our main goal is to relate this notion to the notion of the intrinsic core, and discuss some interesting phenomena pertaining to linear closure of convex sets. 

\subsection{Linear closure and the intrinsic core}\label{sec:linclosure}

We say that a convex subset $C$ of a real vector space $X$ is \emph{linearly closed} if for every $(x,y)\subseteq C$ we have $[x,y]\subseteq C$. The \emph{linear closure} $ \lcl C$ is the smallest convex linearly closed subset of $X$ that contains $C$. In a finite dimensional space the linear closure of a convex set $C$ consists of all line segments contained in $C$ together with their endpoints, but this is not the case in the infinite-dimensional setting (see \cite{KleeLin,nikodym} for a detailed discussion on the fascinating differences between these two notions; note also that it may happen that $\lin C \neq \lin \lin C$). We hence distinguish between the (strong) linear closure $\lcl C$ and the (weak) linear closure $\lin C$.


There is a related notion of \emph{linearly accessible points} discussed \cite{Holmes,cones}: a point $x\in X$ is linearly accessible from $C$ if and only if there is some $y\in X$ such that $(x,y)\subseteq C$. Since our notation allows $x=y$ in this case, this notion coincides with the weak linear closure $\lin C$ (cf. \cite{DangEtAl}). It appears that some references do not allow $x=y$ in the definition of linearly accessible points, and in this interpretation the set of linearly accessible points coincides with our definition $\lin C$ whenever $C$ contains at least 2 points (see \cite{Holmes}). 

We can then define the (weak) \emph{algebraic boundary} of a convex subset $C$ of a real vector space $X$ as the set of all endpoints of the line segments that are contained in $C$, that is,
\[
\lbd C = \{x\in X\,|\, \exists u  \in X: (x,x+u) \subseteq C, (x,x-u)\cap C = \emptyset\}.  
\]
Note that $\lbd C$ is called the set of \emph{outer points} in \cite{InnerStructure}, also this reference defines such points for general (not necessarily convex) subsets of $X$.

\begin{proposition}\label{prop:calcintcore}
Let $C_1$, $C_2$ and $C$ be convex subsets of a real vector space $X$. Then we have the following relations between the weak linear closure and the intrinsic core.
\begin{enumerate}[label=(\roman*)]
\item $\icr C = C \setminus \lbd C$. \label{calc:acldecomp}
\item $\lin C = C\cup \lbd C = \icr C \cup \lbd C.$ \label{calc:aclandabd}
\item If $x\in \icr C$ and $y\in \lin C$, then there exists $z\in \icr C$ such that $x\in (z,y) \subseteq \icr C$. In particular, $[x,y)\subset \icr C$. \label{calc:bdsegment}
\item If $\icr C\neq \emptyset$, then  $C\subseteq \lin \icr C$, $\icr(\lin C)  = \icr C$ and $\lin C = \lin \icr C$.\label{calc:aclicr}
\item We have $\aff C = \aff \lin C $, and if $\icr C \neq \emptyset$, then $\aff \icr C = \aff C$. \label{calc:afficr}
\end{enumerate}
\end{proposition}
\begin{proof}

To show \ref{calc:acldecomp}, let $x\in C$ and observe that  $x\in \lbd C$ if and only if there is a line segment with the endpoint $x$ that can not be continued beyond $x$, which is equivalent to $x\notin \icr C$. 

To show \ref{calc:aclandabd}, first observe that the second equality follows from \ref{calc:acldecomp}: $C\cup \lbd C = (C\setminus \lbd C) \cup \lbd C = \icr C \cup \lbd C$. For the first relation, observe that $\lbd C\subseteq \lin C$ and $C\subseteq \lin C$, hence we only need to show that $\lin C\subseteq C\cup \lbd C$.  If $x\in \lin C$, and $x\notin C$, then $x$ must be the endpoint of some segment $(x,x+u)$ contained in $C$. Since $x\notin C$, the extension of this segment beyond $x$, that is, $(x,x-u)$ must have an empty intersection with $C$. We conclude that $x\in \lbd C$.

For \ref{calc:bdsegment} let $x\in \icr C$, $y\in \lin C$. If $y\in C$, then we are done by Proposition~\ref{prop:basicsegment}. Otherwise there must exist some $u\in C$ such that $u\neq y$ and  $(y,u]\subseteq C$.

Since $x\in \icr C$, there exists $v\in C$ such that $x\in (u,v)$. Then by Proposition~\ref{prop:technical} for any point $z'\in (x,y)$ there exits $w\in (v,y)$ such that $z'\in (u,w)$. Now fix any $p\in (z',w)\subseteq (u,w)$. Applying Proposition~\ref{prop:technical} again, this time to $y,v,w,u$ and $p$ (see Fig.~\ref{fig:technical04}),
\begin{figure}[ht]
    \centering
    \includegraphics[width=0.45\textwidth]{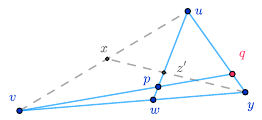}
    \caption{Applying Proposition~\ref{prop:technical} twice in the  proof of Proposition~\ref{prop:calcintcore} \ref{calc:bdsegment}.}
    \label{fig:technical04}
\end{figure}
we conclude that there exists $q\in (u,y)$ such that $p\in (v,q)$. Now $q\in (u,y)\subseteq C$, $v\in C$ and $p\in (v,q)\subseteq C$. Since $z'\in (u,p)$, with $u\in C$, and hence $z'\in C$, by the arbitrariness of $z'$ this proves that $[x,y)\subseteq C$. Now for any point $r\in (x,y)$ we have 
$r\in C$, hence, by Proposition~\ref{prop:basicsegment} we obtain $[x,r)\subseteq \icr C$, and hence $[x,y)\subseteq \icr C$, and also by the same proposition there exists some $z\in \icr C$ such that $x\in (z,r)\subset (z,y)$. 

For \ref{calc:aclicr}, we begin with $\lin C = \lin \icr C$. Since $C \subset \icr C$, we only need to prove that $\lin C\subseteq \lin \icr C$. Take any $y\in \lin \icr C$. If $y\in \icr C$, then $y\in C \subset \lin C$ and we are done. If $y\in \lin \icr C$, take any $x\in \icr C$. By \ref{calc:bdsegment} we have $[x,y)\subseteq \icr C\subset C$. Hence $y$ must be in $\lin C$. 

Observing that $C \subset \lin C$, the relation $\lin C = \lin \icr C$ that we just proved yields $C \subseteq \lin \icr C$.  It remains to prove $\icr (\lin C)=\icr C$. Since  $\icr C\subseteq \icr(\lin C)$ we only need to prove that $\icr(\lin C)\subseteq \icr C$. Take any $y\in \icr (\lin C)$. Since $\icr C\neq \emptyset$, there exists some $x\in \icr C$. Since $y\in \lin C$, by \ref{calc:bdsegment} there exists $z\in \icr C$ such that $x\in (y,z) \subseteq \icr C$, concluding the proof of this item.


 For \ref{calc:afficr} first note that $\aff C$ contains all lines with two different points on $C$, then $\lin C \subseteq \aff C$, implying that $\aff C=\aff(\lin C)$. Using \ref{calc:aclicr}, $C\subseteq \lin(\icr C)$, then $\aff C\subseteq \aff(\lin(\icr C))=\aff(\icr C)$. The conclusion follows because by other hand, $\aff(\icr C)\subseteq \aff C$. \end{proof}


Note that in \cite{DangEtAl}  the relation \ref{calc:bdsegment} is shown using separation theorems (see Theorem~3.7), however our proof is much more elementary.

In \cite[Theorem~2.1]{InnerStructure} it was shown that  $M\subseteq X$ is an affine subspace if and only if $M$ is convex, $\icr M = M$ and $\lbd M = \emptyset$. Indeed we observed earlier that for a linear subspace $M$ we have $\icr M = M$,  $\lbd M  \subset \aff M = M $ and $M = \icr M  =  M \setminus \lbd M$. 

Since in general $\lin C \neq  \lin \lin C \neq \lcl C$, it makes sense to define $n$-th (weak) linear closure, and we can do this for any ordinal in the following way (see \cite{KleeLin}):
\[
\lin^\beta C = \begin{cases}
\lin \lin^{\beta-1} C & \text{if } \beta-1 \text{ exists,}\\
\cup_{\alpha<\beta} \lin^\alpha C & \text{if } \beta \text{ is a limit ordinal.}
\end{cases}
\]
It was shown in \cite{nikodym} that there exists a convex set for which $\lin^\alpha C \neq \lin^{\alpha-1} C$ for all ordinals $\alpha<\omega_1$ (where $\omega_1$ is the first uncountable ordinal). At the same time, $\lcl C = \lin^{\omega_1} C$ (see \cite{nikodym}).

We can define the $n$-th (weak) linear boundary as ${\lbd}^n C = \lin^n C\setminus \icr C$, and likewise the strong linear boundary as $\lcl C \setminus \icr C$. 

The relation $\icr(\lin C)  = \icr C$ fails to hold generically in infinite-dimensional spaces, this is due to the existence of proper convex sets that are \emph{ubiquitous}, that is, such that their linear closure coincides with the entire space.

\begin{theorem}[see \cite{kleepart1}] The linear space $X$ is inifinte dimensional if and only if $X$ contains a proper convex subset $C$ such that $\lin C = X$. 
\end{theorem}
The proof of this theorem is based on an explicit construction of such set, which we study in the next example. This example was publishedin \cite{kleepart1} and according to our Referee, can be attributed to M.~M.~Day.

\begin{example}\label{eg:ubiquitous}

Any linear vector space has a Hamel basis that can be totally ordered, moreover, if the space is infinite-dimensional, we can choose the order in a way that there isn't a greatest element. We can then represent any element as a finite linear combination of the basic points. There will be one coordinate that is maximal with respect to the total order, and we can define the set $C$ as the set of all points with the positive last coordinate. This set is convex, and its linear closure constitutes the entire space $X$: indeed, any point $x\in X$  is either in $C$, or is the endpoint of some line segment in $C$ obtained by adding a small positive coordinate with a higher index. Somewhat counterintuitively, the complement of this set is also convex and ubiquitous (following the same logic).

It is easy to see why the intrinsic core of such a set is empty: take any element $x$ in this set, and then pick $y$ with a higher positive coordinate. Evidently the line through $x$ and $y$ can not be extended beyond $x$, so $x$ is not in the intrinsic core. Hence $\icr C= \emptyset$. 

When it comes to the facial structure of this set, observe that for any $x,y\in C$ there exists $z\in C$ such that $x\in (y,z)$ if and only if the `index' of the last coordinate of $y$ is no larger than the index of the last coordinate of $x$. Therefore, the minimal face of $x$ is the subset of $C$ that contains all elements with the last coordinate not larger than $x$. This in particular means that the set $C$ is the union of the chain of all minimal faces. 

This property of chain of minimal faces of points is not always satisfied. Take $X$ a vector space with an uncountable Hamel basis, and define the convex set $C$ as the set of all the elements from $X$ with non-negative coordinates. The faces of $C$ are the cones spanned by subsets of the Hamel basis, while minimal faces of any point in $C$ is the cone spanned by finite subsets of the Hamel basis. The only faces which are reunions of chains of minimal faces of points in $C$ are the cones spanned by finite or countable subsets of the Hamel basis. Accordingly, all the faces spanned by infinite uncountable subsets of the Hamel basis (including $C$) cannot be represented as the union of a chain of minimal faces of points of $C$. We thank our referee for suggesting this example.

\end{example}

Recall how in finite-dimensional case faces of closed convex sets are always closed. It appears that this phenomenon is specific to the linear  and not topological closure.

\begin{proposition} If $C$ is a linearly closed convex subset of a vector space $X$, then every face of $C$ is linearly closed.
\end{proposition}
\begin{proof} Let $C$ be an linearly closed subset of $X$, and let $F$ be an arbitrary face of $C$. Our goal is to show that for any $(x,y)\subseteq F$ we have $[x,y]\subseteq F$. 

Let $(x,y)\subseteq F\subseteq C$. Since $C$ is linearly closed, $x,y\in C$. Now pick any  $z\in (x,y)$. Since $z\in F$, by the definition of a face we must have $x,y\in F$. 
\end{proof}

The next example shows that faces of (topologically) closed convex sets do not need to be closed. This example is used in \cite{hitchhiker} to demonstrate that the convex hull of a compact set may not be compact (and even closed).

 \begin{example}[{Example~{5.34} from \cite{hitchhiker}}]\label{counterconj}
Let $A\subset \ell_2$ be defined as $A=\{0,e_1,\frac{e_2}{2},\dots, \frac{e_n}{n},\dots\}$, where $e_i$ is the sequence of all ones except for $1$ in the $i$-th position.  Let $F=\co A$ and $C=\cl \co A$ be the $l_2$-closure of the convex hull of $A$. We will show that $F\unlhd C$. Note that $F$ is not closed (but $A$ is compact, see \cite{hitchhiker}), while $C$ is closed by the definition. Also both $F$ and $C$ are convex and $F\subset C$. It remains to show that $F$ is a face of $C$. 

Take $x\in F$, and suppose that $y,z\in C$ are such that $x = t y + (1-t) z$ for some $t\in (0,1)$. Since $e_1,e_2,\dots$ is an orthonormal basis of $l_2$, we can write 
\[
y = \sum_{i=0}^\infty \alpha_i e_i, \quad z= \sum_{i=0}^\infty \beta_i e_i,
\]
where $\alpha_i\geq 0$, $\beta_i\geq 0$ for all $i\in \N\cup\{0\}$, and $\sum \alpha_i = \sum \beta_i = 1$. 

Since $x\in F$, there exists a finite index set $I\subset \N\cup \{0\}$ such that  
\[
x = \sum_{i\in I } \gamma_i e_i, \quad \gamma_i \geq 0, \sum_{i\in I} \gamma_i = 1.
\]
From $ t y + (1-t) z - x =0$ we have 
\[
t \alpha_i + (1-t) \beta_i  = 0 \quad \forall i \notin I,
\]
hence, $y$ and $z$ are finite convex combinations of points in $A$, and so $y,z\in F$. Hence by the definition of a face $F\lhd C$. 

\end{example}

In \cite{zbMATH03075149}, there is an interesting example of two convex and linearly closed sets whose convex hull is not linearly closed.

Thinking in terms of linear closure helps further refine the relations between faces and the intrinsic core. 

\begin{proposition}\label{prop:Fabd} Let $D$ be a convex subset of a real vector space $X$, and let $F$ be a proper face of $D$ (that is, $F$ is nonempty and does not coincide with $D$). Then $F\subseteq \lbd D$. 
\end{proposition}
\begin{proof}
Under the assumptions of the proposition, suppose that the conclusion is not true. Then there exists a point $x\in F\setminus \lbd C$. By Proposition~\ref{prop:calcintcore}~\ref{calc:acldecomp} we then have $x\in \icr C$. By Corollary~\ref{cor:minfaceicr} we must have $D = F_{\min}(x,D)\subseteq F$, which contradicts the assumption that $F$ is a proper face of $D$.
\end{proof}

Notice that $A\subseteq B$ does not yield $\icr A \subseteq B$, even in the finite-dimensional case. For example, letting $A = \{x\}$ and $B = [x,y]$, where $x$ and $y$ are distinct elements of some vector space $X$, results in $A \subset B$, but $\icr A = \{x\}\notin (x,y) = \icr B$. We can however obtain meaningful relations between the intrinsic cores of certain structured subsets of convex sets, as shown next. 

\begin{proposition} Let $F$ be a proper face of $D$, where $D$ is a convex subset of a real vector space $X$. The set $D\setminus F$ is convex; moreover, $\icr (D\setminus F) = \icr D$. 
\end{proposition}
\begin{proof} To show that $D\setminus F$ is a convex set, let $x,y \in D\setminus F$. If there is a $z\in (x,y)$ such that $z\in F$, then by the definition of a face we must have $[x,y]\subseteq F$, a contradiction with the original choice of $x$ and $y$. 

To show that $\icr D = \icr (D\setminus F)$, we first demonstrate that $\icr D \subseteq \icr (D\setminus F)$. Take any $x\in \icr D$. From Proposition~\ref{prop:Fabd} we know that $x\in D\setminus F$. Now pick any $y\in D\setminus F\subset D$. Since $x\in \icr D$, by Proposition~\ref{prop:calcintcore}~\ref{calc:bdsegment} there exists $z\in \icr D$ such that $x\in (y,z)\subseteq \icr D\subseteq D\setminus F$ (applying Proposition~\ref{prop:Fabd} again). Since $y\in D\setminus F$ was arbitrary, we deduce that $x\in \icr (D\setminus F)$. 

It remains to show that $\icr (D\setminus F)\subseteq \icr D$. Take any $x\in \icr (D\setminus F)$, and let $y\in D$. If $y\in D\setminus F$, then there exists a $z\in D\setminus F \subset D$ such that $x\in (y,z)$, so we only need to deal with the situation when $y\in F$. In this case $x\neq y$. Let $u\in (x,y)$. If $u\in F$, then by the definition of a face $x\in F$, which is impossible. We conclude that $u\in D\setminus F$. Then there exists $v\in D\setminus F$ such that $x\in (u,v) \subseteq (y,v)\subset D$. We conclude that $x\in \icr D$. 
\end{proof}

\subsection{Separation and support points}

Recall (see \cite{Holmes}) that a hyperplane in a real vector space $X$ is a maximal proper affine subspace of $X$, or equivalently a level set of some nontrivial linear functional $\varphi:X\to \R$, 
\begin{equation}\label{eq:hyperplane}
H = \{x\in X\, |\, \varphi(x) = \alpha\}.
\end{equation}

A hyperplane $H$ defined as in \eqref{eq:hyperplane} is said to \emph{support} a convex set $C$ if $\varphi(x) \leq \alpha$ for all $x\in C$ (alternatively $\varphi(x)\geq \alpha$ for all $x\in C$), where the equality is satisfied for some $x\in C$.

Two convex sets $A$ and $B$ are \emph{separated} by a hyperplane $H$ if they lie in the two different (linearly closed) subspaces defined by the hyperplane $H$, so that, for instance, 
\[
\varphi(x) \leq \alpha \quad \forall x\in A, \qquad
\varphi(x) \geq \alpha \quad \forall y\in B.
\]
The sets $A$ and $B$ are \emph{properly separated} by $H$ if they are separated by $H$ and do not lie in their entirety on $H$, so $(A\cup B)\setminus H \neq \emptyset$. This is equivalent to the existence of a linear functional $\varphi:X\to \R$ such that 
\[
\sup_{x\in A} \varphi(x) \leq \inf_{y\in B} \varphi(y), \quad 
\inf_{x\in A} \varphi(x) < \sup_{y\in B} \varphi(y).
\]

Hyperplane separation in infinite-dimensional vector spaces is hinged on the following statement (see \cite[I.2.B]{Holmes}). 

\begin{lemma}[Stone] Let $A$ and $B$ be disjoint convex subsets of $X$. Then there exist complementary convex sets $C$ and $D$ such that $A\subseteq C$ and $B\subseteq D$. 
\end{lemma}

The proof is based on using Zorn's lemma to build a maximal convex set that contains $A$ but not $B$ and showing that the complement must be convex. 

It can be further shown that the intersection of weak linear closures of two nonempty complementary convex sets is either the entire space or a hyperplane (see \cite[I.4.A.]{Holmes}). In the former case  there is no hyperplane separating these complementary convex sets. A sufficient condition for the separation of two disjoint convex sets is that at least one of them has a nonempty core (see \cite[I.4.B.]{Holmes} and \cite[5.61]{hitchhiker}).

\begin{theorem}[Theorem~5.61 from \cite{hitchhiker}] Two disjoint nonempty convex sets can be properly separated by a nonzero linear functional provided one of them has a nonempty core.
\end{theorem}

This theorem can be generalised as shown in the following corollary (both these results can be proved using the Hanh-Banach extension theorem).

\begin{corollary}[Corollary~5.62 from \cite{hitchhiker}]\label{cor:sep} Let $A$ and $B$ be two nonempty disjoint convex subsets of a vector space $X$. If there exists a vector subspace $Y$ including $A$ and $B$ such that either $A$ or $B$ has nonempty core in $Y$, then $A$ and $B$ can be properly separated by a
nonzero linear functional on $X$.
\end{corollary}

It is tempting to assume that it may be enough for the intrinsic core of one of the disjoint convex sets to be nonempty to achieve proper separation. However this is not true: for instance, let $A$ be the set considered in Example~\ref{eg:ubiquitous}, and suppose $B$ is any single point in the complement of $A$, meaning that $B$ has a nonempty intrinsic core. Notice that any hyperplane supporting some convex set should also support its linear closure. In our case the linear closure of $A$ is the entire space, and hence the only linear functional supporting $A$ is the trivial one (whose kernel coincides with $X$) that a) doesn't define a hyperplane and b) doesn't properly separate $A$ and $B$. However if both sets have nonempty intrinsic cores, this is enough to ensure proper separation (see Proposition~\ref{prop:icrsep} below).

A point $x\in C$ is called a \emph{support point} if it belongs to some supporting hyperplane. A support point is \emph{proper} if it lies in a supporting hyperplane that does not contain all of $C$ (see \cite{Holmes}). It is well-known that whenever $\icr C \neq \emptyset$, the proper support points of a convex set $C$ are exactly $C\setminus \icr C$. We provide a proof for completeness. 

\begin{proposition} If $\icr C \neq \emptyset$ for some convex subset $C$ of a real vector space $X$, then every $x\in C\setminus \icr C$ is a support point of $C$.
\end{proposition}
\begin{proof} Suppose that a convex set $C$ has nonempty intrinsic core, and that $x\in C\setminus \icr C$. Our goal is to show that there exists a linear functional $\varphi:X\to \R$  such that 
\[
\sup_{u\in C}\varphi(u) \leq \varphi(x),  
\]
and moreover there is some $y\in C$ such that $\varphi(y)<\varphi(x)$, so the separation is proper.

We first prove this claim under the assumption that $x=0$. Let $F = F_{\min}(0,C)$. Since $x\notin \icr C$, $F$ is a proper face of $C$. By Proposition~\ref{prop:Fabd} the set $C\setminus F$ is convex, moreover, $\icr C = \icr (C\setminus F)$.
By Proposition~\ref{prop:calcintcore}~\ref{calc:afficr} we have $\aff \icr C = \aff C$, and together with $\icr C = \icr (C\setminus F) \subseteq C\setminus F \subset C$ this yields $\aff C = \aff (C\setminus F)$. By Definition~\ref{def:affine} of the intrinsic core, $\icr (C\setminus F) $ coincides with the core of $C\setminus F$ with respect to $\aff C$. Notice that since $0\in C \subset \aff C$, the affine hull $\aff C$ is actually a linear subspace of $X$. We can now apply Corollary~\ref{cor:sep} to $A = F$, $B = C\setminus F$ and $Y = \aff C$. There exist a linear functional $\varphi:X\to \R$  such that 
\begin{equation}\label{eq:312134234}
\sup_{u\in C\setminus F}\varphi(u)  \leq \inf_{v\in F}\varphi(v);\qquad
\inf_{u\in C\setminus F}\varphi(u)  < \sup_{v\in F}\varphi(v).
\end{equation}

Take any $x\in \icr C$. By Proposition~\ref{prop:calcintcore}~\ref{calc:bdsegment} we have $(0,x)\subseteq \icr C\subseteq C\setminus F$. Hence for any $t>0$
\[
t \varphi(x) = \varphi(t x ) \leq \varphi(0) = 0,
\]
and we conclude that 
\[
\sup_{u\in C\setminus F}\varphi(u) = 0.
\]
Now take any $p\in F$. Since $0\in \icr F$, there exists $q\in F$ such that $0 \in (p,q)$. It follows from $\varphi(q), \varphi(q)\geq 0$ and $\varphi(0) = 0$ that $\varphi(p) = \varphi(q) = 0$. 

Using all this knowledge, we rewrite \eqref{eq:312134234} as 
\begin{align*}
\sup_{u\in C}\varphi(u)  \leq \varphi(0) = 0;\qquad  \inf_{u\in C}\varphi(u)  < 0. 
\end{align*}

We have found a linear functional $\varphi:X\to \R$ such that this linear functional properly separates $0$ from $C$. 

We can now return to the case when $x\neq 0$. Let $D = C-\{x\}$. Then by Corollary~\ref{cor:addpoint} we have $\icr D = \icr C + \{x\}\neq \emptyset$; moreover, since $0\in D$, $Y = \aff D$ is a linear subspace. Observe also that $0\notin \icr D$. By our earlier result, there  exists a linear functional $\varphi:X\to \R$ such that
\begin{align*}
\sup_{u\in D}\varphi(u)  \leq \varphi(0) = 0;\qquad  \inf_{u\in D}\varphi(u)  < 0. 
\end{align*}
Observe that $u\in D$ iff $u= v-x$ for some $v \in C$, therefore this is equivalent to
\begin{align*}
\varphi(v) & \leq \varphi(x) \quad \forall v \in C;\\
\varphi(v') & < \varphi(x) \quad \text{ for some } v'\in C\setminus \{x\}, 
\end{align*}
hence we found a linear functional $\varphi:X\to \R$ and a constant $\alpha = \varphi(x)$ such that $\varphi$ properly separates $\{x\}$ from  $C$. 
\end{proof}

\begin{proposition}\label{prop:icrsep}
If $A$ and $B$ are convex subsets of a real vector space $X$ such that $\icr A, \icr B \neq \emptyset$ and moreover $\icr A \cap \icr B = \emptyset$, then $A$ and $B $ can be properly separated.
\end{proposition}

We first need the following technical claim.

\begin{proposition}\label{prop:tech324} Let $C$ be a convex subset of a real vector space $X$. If $\icr C \neq \emptyset$ then the positive hull of $C$,
\[
D = \{t x\, |\, x\in C, t>0\}
\]
has nonempty intrinsic core, moreover, $\icr C \subseteq \icr D$. 
\end{proposition}
\begin{proof} Under the assumptions of the proposition, take any $x\in \icr C$ and $y\in D$. There exists some $t>0$ such that $y = t y'$ for some $y'\in C$. Since $x\in \icr C$, there must exist $z'\in C$ such that $x\in (y',z')$, and explicitly there is some $s\in (0,1)$ such that 
\[
x =  s y'+ (1-s)z'.
\]
Now if $t\geq 1$, observe that 
\[
x = \frac{s}{t} t y' + \left(1-\frac{s}{t}\right)  \frac{(1-s)t }{t-s} z' = \frac{s}{t} y +\left(1-\frac{s}{t}\right) z, 
\]
where 
\[
z =  \frac{(1-s)t }{t-s} z' \in D,
\]
since the coefficient at $z'$ is positive. Moreover, since $s/t\in (0,1)$, we conclude that $x\in (y,z)\subseteq D$. 

It remains to consider the case when $t<1$. In this case notice that $v: = 2 y'\in D$, and also $y'\in (y,v)$. Applying Proposition~\ref{prop:technical} to the points $y', z', y, v$ and $x$, we deduce that there exists $z\in (z',v)\subseteq D$ such that $x\in (y,z)\subseteq D$. 

We conclude that $\icr C \subseteq \icr D$.  
\end{proof}

\begin{proof}[Proof of Proposition~\ref{prop:icrsep}] Under the assumptions of the proposition, let $C:= A-B$. By Proposition~\ref{prop:sum} we have $\icr C = \icr A - \icr B$. Moreover, since $\icr A \cap \icr B = \emptyset$, we have $0\notin \icr C$. If $0\in \aff \icr C$, then by Corollary~\ref{cor:sep} $\icr C$ can be properly separated from $\{0\}$. Otherwise let $D$ be the positive hull of $C$. From Proposition~\ref{prop:tech324} we know that $\icr C \subseteq \icr D$, moreover it is evident that $0\notin D$, but $0\in \aff D$. Therefore $\{0\}$ can be separated properly from $\icr D$. Since $D$ is the positive hull of $C$, this means that the same linear functional properly separates $C$ from $\{0\}$.

Since we now found that there exists a linear functional $\varphi:X\to \R$ that properly separates $\icr C$ from $0$, this yields  
\[
\varphi(x)\leq 0 \quad \forall x\in  \icr C, \quad \exists x_0 \in \icr C, \varphi(x_0)< 0,
\]
Since $C\setminus \icr C  \subseteq \lbd \icr C$ (by Proposition~\ref{prop:calcintcore}  \ref{calc:aclandabd} and \ref{calc:aclicr}), we also have  
\[
\varphi(x)\leq 0 \quad \forall x\in  C, \quad \exists x_0 \in C, \varphi(x_0)< 0.
\]
Recall that $C = A-B$. We hence have 
\[
\sup_{u\in A}\varphi(u) \leq \inf_{v\in B} \varphi(v), \quad \inf_{u\in A} \varphi(u) < \sup_{v\in B} \varphi(v), 
\]
and so $A$ and $B$ can be properly separated. 
\end{proof}

Recall that a face $F\unlhd C$ is \emph{exposed} if there exists a supporting hyperplane $H$ to $C$ such that $F = C\cap H$. 

In the finite-dimensional setting whenever a face is not exposed, it is always \emph{eventually exposed}: we can choose a hyperplane exposing some face of $C$, then expose a face of that face and so on, until we reach the required face. It is unclear whether this result generalises in a sensible way to the infinite-dimensional setting (under the assumption that $\icr C \neq \emptyset$). 

We also note that some properties of exposed faces that naturally hold in finite dimensions are not true in infinite-dimensional spaces, for instance, the intersection of two exposed faces is not necessarily exposed \cite{ShortNote} in a Banach space, and in \cite{MinExp} sufficient conditions are obtained for the intersection of exposed faces to be exposed.

\section{Summary}\label{sec:summary}

We have provided an overview of the definitions and properties of intrinsic core, intentionally focusing on the restricted setting of real vector spaces without topological structure. We note once more that the intrinsic core appears in the literature under different names, including the set of relatively absorbing points, the pseudo-relative interior and the set of inner points (see the discussion at the beginning of Section~\ref{sec:intrinsic}). 

We have discussed four equivalent definitions of the intrinsic core, via line segments (Definition~\ref{def:segments}), the cone of feasible directions (Definition~\ref{def:feasibled}), as a core with respect to the affine hull (Defintion~\ref{def:affine}) and in terms of minimal faces (Definition~\ref{def:viaminface}). We have also made explicit the differences in notation and discrepancies in implicit assumptions used by different authors that should aid future work involving intrinsic cores. 

We have provided an extensive collection of properties and calculus rules of intrinsic cores, and discussed in much detail the relation between the intrinsic core and algebraic boundary. We have also provided an overview and elementary proofs of separation results pertaining to this purely algebraic setting. 




\section*{Acknowledgements}

The authors are indebted to the two referees, who not only carefully read an earlier draft of our paper, but also generously provided many insights, including examples and references, that significantly improved both the quality of this work and our understanding of the subject matter.

We are also grateful to the Australian Research Council for financial support provided by means of the Discovery Project ``An optimisation-based framework for non-classical Chebyshev approximation'', DP180100602.

\bibliographystyle{plain}

\bibliography{icr-refs}

\end{document}